\documentclass[11pt,final]{article}

\usepackage{geometry}
 \fontfamily{cmr}

\usepackage{graphicx}
\usepackage{amssymb}
\usepackage{epstopdf}
\usepackage{amsthm}
\usepackage{latexsym}
\usepackage{graphicx}
\usepackage[usenames,dvipsnames]{color} 
\usepackage{tikz,amsmath,amsfonts}
\usepackage{mathtools}
\usepackage{ifthen}
\usepackage{caption}

\usepackage{tikz}
\usepackage{tikz-cd}
\usetikzlibrary{cd}
\usetikzlibrary{decorations.pathmorphing}
\usepackage{color}
\definecolor{pur}{RGB}{186,146,162}
\usepackage{marginnote}
\pgfdeclarelayer{background}
\pgfsetlayers{background,main}

\usepackage{titlesec}
\titlelabel{\thetitle.\quad }

\theoremstyle{definition}

\newtheorem{defn}{\normalfont\scshape Definition}
\newtheorem{prop}{\normalfont\scshape Proposition}

\newtheorem{thm}{\normalfont\scshape Theorem}

\newcommand{\op}{\text{op}}

\newcommand{\Set}{\textbf{Set}}

\newcommand{\CG}{\textit{CG}}
\newcommand{\CW}{\textit{CW}}
\newcommand{\Grpd}{\textit{Grpd}}
\newcommand{\Kan}{\textit{Kan}}
\newcommand{\PShv}{\textit{PShv}}
\newcommand{\s}{\textit{s}}
\newcommand{\Shv}{\textit{Shv}}
\newcommand{\Sm}{\textit{Sm}}
\newcommand{\Spc}{\textit{Spc}}
\newcommand{\Spectra}{\textit{Spectra}}
\newcommand{\sSet}{\textit{sSet}}
\newcommand{\Top}{\textit{Top}}

\newcommand{\A}{\mathbb{A}}
\newcommand{\U}{\mathbb{U}}
\newcommand{\V}{\mathbb{V}}
\newcommand{\W}{\mathbb{W}}

\newcommand{\C}{\mathcal{C}}
\newcommand{\D}{\mathcal{D}}
\renewcommand{\H}{\mathcal{H}}
\newcommand{\M}{\mathcal{M}}
\newcommand{\N}{\mathcal{N}}

\newcommand{\Fun}{\text{Fun}}
\newcommand{\Hom}{\text{Hom}}
\newcommand{\Ho}{\text{Ho}}
\newcommand{\hyp}{\text{hyp}}
\newcommand{\Map}{\text{Map}}
\newcommand{\Nis}{\text{Nis}}

\newcommand{\Zar}{\text{Zar}}

\newcommand{\hatS}{\widehat{\mathcal{S}}}

\renewcommand{\Pr}{\mathcal{P}\hspace{-0.15em}\textit{r}}

\title{Constructing The Unstable Motivic Homotopy Category Using $(\infty,1)$-Categories}
\author{Thomas Brazelton}
\date{2018}

\begin{document}
\maketitle

\begin{abstract}
\noindent In this expository paper, we provide an overview of the construction of the unstable motivic homotopy category $\H(S)$ over a base scheme $S$. We first provide a modern construction using the theory of model categories and Bousfield localization. From the perspective that each model category has an underlying $(\infty,1)$-category, we proceed to construct $\H(S)$ in two ways, one from the perspective taken in a paper by Bachmann and Hoyois \cite{BH17}, and another from the universal algebra perspective of Robalo \cite{R15}. We finally conclude by demonstrating that these two constructions exhibit an equivalence of $(\infty,1)$-categories.
\end{abstract}

\section*{Acknowledgements}
Thank you to Emily Riehl for supervising this expository paper, and to Kirsten Wickelgren for suggesting the topic and providing me with references and support. Additionally, thank you to Tom Bachmann for his helpful correspondence. Finally, thank you to all the references I have cited, particularly \cite{AE16}, whose exposition I followed closely. All of the writing here is expository, so any mistakes are my own, and any credit should be attributed to the authors cited.

\section{Introduction}
Voevodsky \cite{V98} draws a portrait of the abstract theory of (stable) homotopy as follows:\\

We begin with a category $\C$ in which we would like to study homotopy. A priori, $\C$ need not have any nice properties that we desire in a category, so we usually pass to a category $\Spc$ of ``spaces,'' which contains $\C$ and has internal-homs, as well as small limits and colimits. Now we define a class of weak equivalences in $\Spc$. We choose classes of cofibrations and fibrations in order to get a model structure on this category, and obtain the \textit{(unstable) homotopy category} $H$. This category should be pointed, i.e. have a zero object, so that we can talk about suspension functors in a meaningful way. If we were to stabilize with respect to these suspension functors, which satisfy some nice natural properties, we would obtain a category which is additive and triangulated, but lacks infinite coproducts. In order to avoid this obstruction, we utilize the theory of spectra for our inversion. This produces a category called the \textit{stable homotopy category} $SH$. These constructions are summarized in the following diagram, where each map universally adjoins the structure described in the paragraph above
\[
\C \to \Spc \to H \to SH.
\]

As an example, we could begin with $\CW$, the category of CW-complexes, and obtain
\[
\CW \to \CG \to \Ho(\Top_\ast) \to \Ho(\Spectra),
\]

where $\CG$ is the category of compactly generated spaces.\\

This paper will discuss the approaches to develop the homotopy category in the context of $\A^1$-homotopy theory. We will neglect to discuss stabilization, and will instead focus on the particular constructions of the unstable homotopy category, both utilizing model categories and, in a more modern fashion, using $(\infty,1)$-categories.\\

In our setting, we would like to begin with a category of smooth schemes $\Sm/S$ over a base scheme $S$, and ``do homotopy theory.'' In this setting, we would like the affine line $\A^1$ to play the role of the interval, and we would like the following properties to hold:

\begin{enumerate}
\item $\A^1 \times X \to X$ should be an equivalence (that is, $\A^1$ should be contractible)
\item If $X \in \Sm/S$, and $U,V\subseteq X$ are open subschemes such that $U\cup V = X$ in $\Sm/S$, then this should remain true as we pass to the new categories we are creating.
\end{enumerate}

These two key ideas ($\A^1$-invariance and Nisnevich descent) are what we should bear in mind when attempting to construct any meaningful category in which to do homotopy on schemes.\\

Immediately we see that the category $\Sm/S$ does not admit all colimits. This issue can be rectified in a number of ways. One idea is that we could adjoin other varieties, including nonsmooth ones, and extend this category somehow. However we want all colimits, not only finite ones, which is not possible through this construction. The next idea is to look at the category
\[
\PShv(\Sm/S)
\]
of $\Set$-valued presheaves on $\Sm/S$, and consider its full subcategory $\Shv_\Nis (\Sm/S)$ of sheaves in the Nisnevich topology. This category has all small limits and colimits, and thus becomes our category $\Spc$ of spaces. We can then define a model structure and proceed to lay the groundwork for $\A^1$-homotopy theory as in \cite{MV99,V98}.\\

One final idea is to proceed as Dugger does in \cite[\S~8]{D01}, and expand the category $\Sm/S$ by adjoining homotopy colimits subject to certain relations. This yields a model category which is Quillen equivalent to the one produced in \cite{MV99}.\\

However a more modern perspective notes that a model category is a presentation of a more fundamental object: namely an $(\infty,1)$-category. This tells us that we might be better served to focus on the associated $(\infty,1)$-categories, rather than model categories, as the fundamental objects to localize.\\

Formally speaking, for any simplicial model category $\M$, we may form the subcategory of fibrant-cofibrant objects $\M^\circ$, which is enriched over Kan-complexes. We then take its homotopy coherent nerve $N(\M^\circ)$, and call this the $(\infty,1)$-category \textit{presented by $A$}.\\

In Section \ref{background} we recall the theory of model categories, Bousfield localizations, and $(\infty,1)$-categories. In Section \ref{hs-model-cat}, we present the construction of $\M_{\A^1}$, the model category of Morel-Voevodksy spaces, as in \cite{R15}. Its underlying $(\infty,1)$-category is the unstable motivic homotopy category $\H(S)$, which we construct in Section \ref{hs-infty}. In Section \ref{hs-bh}, we present the construction of $\H(S)$ as found in \cite{BH17}. In Section \ref{hs-r}, we present a construction using universal algebra from \cite{R15}. Finally, in Section \ref{hs-equiv}, we demonstrate that these constructions provide an equivalent definition for the unstable homotopy category $\H(S)$.

\section{Background}\label{background}
\subsection{Model, Simplicial, and Combinatorial Categories}
\begin{defn} Let $\M$ be a category. A \textit{model structure} on $\M$ consists of three classes of morphisms $W$ (weak equivalences), $C$ (cofibrations), and $F$ (fibrations) such that
\begin{enumerate}
\item[\textbf{M1}] The class $W$ satisfies the ``two-out-of-three'' property: if any two of $f$, $g$, $g\circ f$ are weak equivalences, then so is the third.
\item[\textbf{M2}] Each class is closed under retracts.
\item[\textbf{M3}] For a diagram of the form \begin{center}\begin{tikzcd}
A\rar\dar["i" left] & E\dar["p" right]\\
X\rar\ar[ur,dashed] & B
\end{tikzcd}\end{center}
if $p\in W\cap F$ (it is an \textit{acyclic fibration}) and $i\in C$, then a lift $X\to E$ may be found to make the diagram commute. Alternatively if $i\in W\cap C$ (it is an \textit{acyclic cofibration}) and $p\in F$ then such a lift exists.
\item[\textbf{M4}] Any map in $\M$ admits factorization systems of the form $(W\cap C, F)$ and $(C,W\cap F)$. This means that any morphism can be expressed as the composition of an acyclic cofibration and a fibration, or as a cofibration and an acyclic fibration.
\end{enumerate}
\end{defn}

\begin{defn} A \textit{model category} is a category with a model structure, as well as all small limits and colimits.
\end{defn}

In particular, the presence of small limits and colimits means that any model category $\M$ has an initial object and a terminal object. If these two objects are isomorphic, then we say that $\M$ is \textit{pointed}.\\

Given a model category $\M$, since we have already labeled a class of weak equivalences $W$, we know which maps to formally invert. We thus define the localization as $\Ho(\M) = \M[W^{-1}]$. This comes with a functor $L:\M \to \Ho(\M)$ which is universal among functors that take weak equivalences to isomorphisms. Additionally for any other category $\N$ the map $\Fun(\Ho(\M),\N) \xrightarrow{-\circ L} \Fun(\M,\N)$ is fully faithful.\\

\textsc{A Note On Inversion:} For an locally small category $\M$ with a class of maps $W$, the formal inversion $\M[W^{-1}]$ may not even be a locally small category. Attempting to form the hom-\textit{sets} $\Hom_{\Ho(\M)}(X,Y)$ may require us to transition into a larger universe in order to form a category. It is a theorem of Quillen that if $\M$ is a model category and $W$ is the class of weak equivalences, then $\M[W^{-1}]$ is in fact a category. When we discuss inversion of $(\infty,1)$-categories, this issue of size will be important to remember.\\ 

For certain model categories, we can localize with respect to a set of maps $I$ of $\M$ (not among the weak equivalences), and moreover we can do so in such a way that highlights a model structure that is ``local'' with respect to this set $I$. We will see this in Section \ref{bousfield}, but in order to make sense of this, we must first develop some more machinery.\\

\begin{defn} A category $\M$ is a (tensored and cotensored) \textit{simplicial category} if
\begin{enumerate}
\item We have mapping spaces which are simplicial sets. That is, we can define $\Map_\M(X,Y) \in \sSet$ for every pair $X,Y\in\M$.
\item We have an action of $\sSet$ on $\M$, denoted $\otimes$.
\item We have an exponential $X^S \in \M$ for every $X\in\M$ and $S\in\sSet$.
\end{enumerate}
With the further conditions that $-\otimes X \dashv \Map_\M(X,-)$ forms a two-variable adjunction of $X^S$ for each $X\in\M$, and that $\Hom_\M (X,Y) \cong \Map_\M(X,Y)_0$.

\end{defn}

\begin{defn}
If $\M$ is a model category, which is also a simplicial category, we say that it is a \textit{simplicial model category} if, for a pullback diagram of the following form
\begin{center}
\begin{tikzcd}
\Map_\M(X,E)\ar[dr,dashed, two heads, "g" above] \ar[drr, "i\circ -" above right, bend left=10]\ar[ddr,"-\circ p" below left, bend right=15]& & \\
& \Map_\M (A,E)\times_{\Map_\M (A,B)} \Map_\M (X,B) \rar \dar  & \Map_\M(A,E)\dar["-\circ p" right]\\
& \Map_\M(X,B)\rar["i\circ -" below] & \Map_\M(A,B)\arrow[ul, phantom, "\lrcorner", near end]\\
\end{tikzcd}
\end{center}
where $i:A\to X$ is a cofibration and $p:E\to B$ is a fibration, then the induced map $g$ is a fibration in $\sSet$ and is moreover a weak equivalence if either of $i$ or $p$ is.
\end{defn}

\begin{defn}\cite[Prop~1.1.5.10]{HTT} If $\M$ is a simplicial model category, then the subcategory $\M^\circ$ of fibrant-cofibrant objects is a simplicial category, and its nerve $N(\M^\circ)$ is an $(\infty,1)$-category, referred to as the \textit{underlying $(\infty,1)$-category of $\M$}.
\end{defn}

\subsection{Bousfield Localization}\label{bousfield}
For an adjoint pair of functors $F: \M\rightleftarrows \N:G$ between model categories, we call this a \textit{Quillen pair} if either of the following equivalent conditions hold:
\begin{enumerate}
\item $F$ preserves cofibrations and acyclic cofibrations
\item $G$ preserves fibrations and acyclic fibrations
\end{enumerate}

A Quillen pair induces functors $LF:\M\to \Ho(\N)$ and $RG:\N\to \Ho(\M)$, called the left and right derived functors respectively, which map weak equivalences to isomorphisms. Moreover, this induces an adjunction
\[
\Ho(\M) \leftrightarrows \Ho(\N).
\]

If the adjunction of derived functors $LF\dashv RG$ is additionally an equivalence of categories, then we say that $F\dashv G$ is a \textit{Quillen equivalence}.\\

\begin{defn}
Let $I$ be a set of maps in a simplicial model category $\M$. We say $X\in\M$ is \textit{$I$-local} if it is fibrant and for any $i:A\to B$ in $I$, we have that the induced morphism $\Map_{\M} (B,X) \to \Map_\M (A,X)$ is a weak homotopy equivalence in $\sSet$.\\

We say that a morphism $f:Y\to Z$ is an \textit{$I$-local weak equivalence} if for every $I$-local object $X$ the induced morphism $f^\ast : \Map_\M(Z,X) \to \Map_\M(Y,X)$ is a weak equivalence. We call this class of maps $W_I$, and note that $W \subset W_I$.\\

We define $F_I$ to be the class of maps satisfying the right lifting property with respect to acyclic cofibrations of the form $W_I \cap C$. Then $(W_I, C, F_I)$ forms a model structure on $\M$. We then define $L_I \M$ to be the category $\M$ with this distinguished model structure, and we call this a \textit{left Bousfield localization} of $\M$ with respect to $I$.
\end{defn}

We will fly through a couple definitions here. We define a \textit{combinatorial category} to be a category which is ``locally presentable'' and ``cofibrantly generated,'' although we will neglect to specify what either of those are. Additionally we define a model category to be \textit{left proper} if pushouts of weak equivalences along cofibrations are weak equivalences. 

\begin{thm}{\cite[A.3.7.3]{HTT}} \label{thm:LurA373} If $\M$ is a left proper combinatorial simplicial model category and $I$ is a set of morphisms in $\M$, then the left Bousfield localization $L_I \M$ exists and inherits a simplicial model structure from $\M$. Moreover the fibrant objects of $L_I \M$ are precisely the $I$-local objects of $\M$.
\end{thm}

These definitions shouldn't add too much additional structure, in fact we have:
\begin{prop}\cite[6.4]{D01}
Any combinatorial model category is Quillen equivalent to one which is both simplicial and left proper.
\end{prop}

\subsection{From Model Categories to $(\infty,1)$-Categories}\label{loc-pres-cat}

Much of the passage to $(\infty,1)$-categories here arises as a result of developments in the theory of \textit{locally presentable} $(\infty,1)$-categories. We provide a very light introduction to this theory here, in order to illuminate the connection between the model categories we will work with in Section \ref{hs-model-cat} and the $(\infty,1)$-categories that we encounter later on.

\begin{defn} An $(\infty,1)$-category $\mathcal{C}$ is \textit{locally presentable} if there exists a combinatorial simplicial model category $A$ such that $\C\simeq N(A^\circ)$.
\end{defn}

\begin{defn} Let $\C$ be an essentially small category. Then $\s\PShv(\C) := [\C^\op , \sSet]$ is a functor category called the \textit{category of simplicial presheaves on $\C$}. 
\end{defn}

\begin{thm}{[Dugger's Theorem]} Every combinatorial model category $A$ is Quillen equivalent to the left Bousfield localization of $\s\PShv(K)_{proj}$, with the projective model structure, for some small category $K$. That is,
\[
L_S \s\PShv(K)_{proj} \xrightarrow{\cong} A.
\]
\end{thm}

So we can think of every locally presentable $(\infty,1)$-category as presented by a reflective subcategory of a category of simplicial sheaves. In particular, we have the following:

\begin{prop} For a simplicial set $\D$ (in particular a quasi-category), the category of \textit{$(\infty,1)$-presheaves} on $\D$ is given as
\[
\PShv(\D) := \Fun(\D^\op, \infty\text{-}\Grpd),
\]
where $\infty$-$\Grpd$ is considered as an $(\infty,1)$-category. We have that $\PShv(\D)$ is locally presentable, and is in fact presented by a global (injective or projective) model structure on simplicial presheaves.
\end{prop}

This leads us to the following general idea: simplicial presheaves are a model for $(\infty,1)$-presheaves.

\section{Unstable Motivic Homotopy Theory Using Model Categories}\label{hs-model-cat}

In this section, we will assume $S$ is a quasi-compact and quasi-separated scheme. We have then that the category $\Sm/S$ of smooth schemes over $S$ is essentially small.\\

The Nisnevich topology is a topology that we may place on the category $\Sm/S$ which is finer than the Zariski topology, but coarser than the \'etale topology. Moreover for any full subcategory $\C \subseteq \Sm/S$, the identity functor induces morphisms of sites
\[
\C_{\text{\'et}} \to \C_{\Nis} \to \C_{\Zar}.
\]

\begin{prop}{\cite[Proposition~3.32]{AE16}} The category $\s\PShv(\Sm/S)$ has a model structure with objectwise weak equivalences (in $\sSet$), objectwise fibrations (also in $\sSet$), and ``projective cofibrations,'' meaning that they have left lifting with respect to acyclic objectwise fibrations. This model structure is called the \textit{projective model structure}, and the category $\s\PShv(\Sm/S)$ is moreover a left proper combinatorial simplicial model category.
\end{prop}

We can now take the Bousfield localization of the site $\s\PShv(\Sm/S)$ with respect to \textit{Nisnevich hypercovers}, and obtain the \textit{Nisnevich-local model category}, denoted:
\[
L_{\Nis} \s\PShv(\Sm/S).
\]

This inherits a left proper combinatorial simplicial structure from $\s\PShv(\Sm/S)$, and is a model category by the theory of Bousfield localization.\\

We also note that the fibrant objects here are those objects which take values in Kan-complexes and satisfy Nisnevich descent.\\

However we cannot pass to the homotopy category just yet, since we are missing a crucial detail. We haven't guaranteed that maps of the form $\A^1 \times_S X \to X$ are weak equivalences. To remedy this, we must perform Bousfield localization again.\\

Let $J$ consist of maps $\A^1 \times_S X \to X$ ranging over representatives $X$ for each isomorphism class in $\Sm/S$. We then localize with respect to $J$, which we denote ``$L_{\A^1}$'':
\[
\Spc_S^{\A^1} := L_{\A^1} L_{\Nis}\ \s\PShv(\Sm/S).
\]

We will also use the notation $\M_{\A^1} = \Spc_S^{\A^1}$, as in Robalo, to keep track of the fact that it was constructed via the use of model categories.

\begin{defn}  {\cite[Definition~3.54]{AE16}} Fibrant objects in $\M_{\A^1}$ are called \textit{$\A^1$-local spaces}. Moreover for any $X\in\s\PShv(\Sm/S)$, we have that $X$ is fibrant in $\M_{\A_1}$ if it
\begin{enumerate}
\item takes values in $\Kan$ (it is fibrant in the projective model structure on $\s\PShv(\Sm/S)$)
\item satisfies Nisnevich descent (it is fibrant in $\L_{\Nis}\s\PShv(\Sm/S)$)
\item $X(U) \to X(\A^1 \times_S U)$ is a weak equivalence in $\sSet$ for all $U\in\Sm/S$.
\end{enumerate}
\end{defn}

\section{Unstable Motivic Homotopy Theory Using $(\infty,1)$-Categories}\label{hs-infty}

We hope to create an $(\infty,1)$-category of Morel-Voevodsky spaces $\H(S)$, as in \cite{R15}, using $(\infty,1)$-categories in each step. Additionally, we would like to do so in such a way that $\H(S)$ is equivalent to the underlying $(\infty,1)$-category $\M_{\A^1}$.\\

We give two constructions of the unstable homotopy category $\H(S)$. The first can be found in \cite{BH17}, and relies on a category of finite-type schemes which is a small $(\infty,1)$-category. We will also see a universal construction, as given in \cite{R15}. Finally, we will demonstrate their equivalence, both with each other and with $\M_{\A^1}$.

\subsection{The Construction of the Unstable Homotopy Category} \label{hs-bh}

Let $\C$ be a small $(\infty,1)$-category with finite coproducts. We denote by $\PShv_\Sigma (\C) \subset \PShv(\C)$ the full subcategory of presheaves which map finite coproducts to finite products. This is known as the \textit{nonabliean derived $(\infty,1)$-category of $\C$}. We then obtain an adjunction
\[
\PShv_\Sigma(\C) \stackrel{\stackrel{L_\Sigma}{\leftarrow}}{\hookrightarrow} \PShv(C).
\]

We consider the topology on $\C$ generated by finite coproduct decompositions, and we denote $\Shv_\sqcup(\C)\subset \PShv(\C)$ the $\infty$-topos of sheaves for this topology.

\begin{defn} An $(\infty,1)$-category is called \textit{extensive} if it admits finite coproducts, if binary coproducts are disjoint, and if finite coproduct decompositions are stable under pullbacks.
\end{defn}

\begin{prop}{\cite[Lemma~2.4]{BH17}} Let $\C$ be a small extensive $(\infty,1)$-category. Then $\PShv_\Sigma (\C) \simeq \Shv_\sqcup (\C)$ (moreover $L_\Sigma$ is left exact and $\PShv_\Sigma (\C)$ is an $\infty$-topos).
\end{prop}

Let $S$ be a quasi-compact quasi-separated scheme. Then $\Sm^{ft}/S$, the category of finitely-presented smooth $S$-schemes, is extensive, so we have that $\PShv_\Sigma(\Sm^{ft}/s) = \Shv_\sqcup(\Sm^{ft}/S)$. Since coproduct decompositions are Nisnevich coverings, we have that
\[
\Shv_\Nis( \Sm^{ft}/S ) \subset \PShv_\Sigma (\Sm^{ft}/S).
\]
We let $\PShv_{\A^1} (\Sm^{ft}/S) \subset \PShv(\Sm^{ft}/S)$ be the full subcategory of $\A^1$-invariant presheaves. We can then define the motivic homotopy $(\infty,1)$-category as
\[
\H(S) = \Shv_\Nis (\Sm^{ft}/S) \cap \PShv_{\A^1} (\Sm^{ft}/S).
\]

This should be, by definition, the pullback along the full inclusions of $(\infty,1)$-categories
\begin{center}
\begin{tikzcd}
\H(S)\arrow[dr, phantom, "\lrcorner", very near start] \rar\dar & \Shv_\Nis (\Sm^{ft}/S)\dar \\
\PShv_{\A^1} (\Sm^{ft}/S) \rar & \PShv(\Sm^{ft} / S).
\end{tikzcd}
\end{center}

\subsection{A Universal Construction of $\H(S)$}\label{Robalo} \label{hs-r}
Let $\U\in\V\in\W$ be universes. We consider $S$ to be a Noetherian $\U$-scheme, and $\Sm^{ft}/S$ to be the category of smooth separated $\U$-small schemes of finite type. We then have that $\Sm^{ft}/S$ is $\V$-small.\\

We can take the nerve functor, and consider the $\V$-small $(\infty,1)$-category $N(\Sm^{ft}/S)$. With the Nisnevich topology, this becomes an $\infty$-site.\\

In order to ensure we have all ($\V$-small) colimits, we take the big $(\infty,1)$-category $\PShv(N(\Sm^{ft}/S)) = \Fun(N(\Sm^{ft}/S)^\op , \hatS)$, where $\hatS$ denotes the $(\infty,1)$-category of all $\V$-small $(\infty,1)$-spaces. This is exactly the free completion of $N(\Sm^{ft}/S)$ with all $\V$-small colimits.\\

We have some analogue of the Yoneda embedding, given as
\[
j: N(\Sm^{ft}/S) \hookrightarrow \PShv(N(\Sm^{ft}/S))
\]
where we identify each scheme with its image via $j$.\\

From [R, p.458] and [L, Prop 4.2.4.4], we can identify $\PShv(N(\Sm^{ft}/S)) \cong N((\s\PShv(\Sm^{ft}/S))^\circ)$, that is we have that $\PShv(N(\Sm^{ft}/S))$ is associated with the underlying $(\infty,1)$-category of $\s\PShv(\Sm^{ft}/S)$.\\

We can restrict to the objects in $\PShv(N(\Sm^{ft}/S))$ which are sheaves with respect to the Nisnevich topology. We then obtain an $\infty$-topos $\Shv_\Nis (\Sm^{ft}/S)$.\\

In an $\infty$-topos of sheaves, we often wish to \textit{hypercomplete}, which is a process that allows us to retain only those sheaves which satisfy descent with respect to hypercovers (in our case Nisnevisch hypercovers). This is some higher analog of our localization with respect to Nisnevich hypercovers. We then obtain
\[
\Shv_\Nis (\Sm^{ft}/S)^\hyp
\]
and finally we localize $\A^1$-invariant maps via the functor $\ell_{\A^1}: \Shv_\Nis (\Sm^{ft}/S)^\hyp \to \H(S)$. Altogether we get

\[
N(\Sm^{ft}/S) \xhookrightarrow{j} \PShv(N(\Sm^{ft}/S))\to \Shv_\Nis (\Sm^{ft}/S) \to \Shv_\Nis (\Sm^{ft}/S)^\hyp \to \H(S).
\]

\subsection{Equivalence of These Constructions} \label{hs-equiv}

Recall our two constructions from \cite{R15}:
\begin{center}
\begin{tabular}{ l l }
$\M_{\A^1} = L_{\A^1} L_{\Nis}\ \s\PShv(\Sm/S)$ & model category\\
$\H(S) = \ell_{\A^1} \Shv_\Nis (\Sm^{ft}/S)^\hyp$ & $(\infty,1)$-category\\
\end{tabular}
\end{center}

The fact that $\H(S)$ is really the underlying $(\infty,1)$-category of $\M_{\A^1}$ follows from our discussion of locally presentable categories in Section \ref{loc-pres-cat}, and the relevant results on localization in \cite[A.3.7]{HTT}.\\

We now need to demonstrate that the constructions from \cite{R15} and \cite{BH17} are equivalent, that is
\[
\H(S) = \ell_{\A^1} \Shv_\Nis (\Sm^{ft}/S)^\hyp \simeq \Shv_\Nis (\Sm^{ft}/S) \cap \PShv_{\A^1} (\Sm^{ft}/S)
\]
is an equivalence of infinity categories. Let $S$ be a Noetherian scheme of finite Krull dimension. Then by \cite[2.2.5]{DAG}, we have that $\Shv_{\Nis}(\Sm^{ft} / S)$ is hypercomplete, i.e. $\Shv_{\Nis}(\Sm^{ft} / S)^\hyp \simeq \Shv_{\Nis}(\Sm^{ft} / S)$. We then look at the diagrams\\
\begin{table}[h]
\begin{center}
\begin{tabular}{c c}

\begin{tikzcd}
\PShv(\Sm^{ft}/S) \rar[shift left, "L_{\A^1}"]\dar[shift right,"L_\Sigma" left]				& \PShv_{\A^1} (\Sm^{ft} / S)\ar[ddd, shift left]	\\
\PShv_\Sigma (\Sm^{ft}/S)\dar[equal]			&					\\
\Shv_{\sqcup} (\Sm^{ft} / S)\dar[shift right, "\ell_\Nis" left]				& \\
\Shv_{\Nis}(\Sm^{ft} / S)\rar["\ell_{\A^1}"]				& 		\H(S) 
\end{tikzcd}
&
\begin{tikzcd}
\PShv(\Sm^{ft}/S)				& \PShv_{\A^1} (\Sm^{ft} / S)	\lar[hook, shift left]	\\
\PShv_\Sigma (\Sm^{ft}/S)\uar[shift right,hook]\dar[equal]			&					\\
\Shv_{\sqcup} (\Sm^{ft} / S)			& \\
\Shv_{\Nis}(\Sm^{ft} / S)\uar[hook,shift right]				& 		\H(S) \ar[uuu, hook, shift left]\lar[hook,shift right]
\end{tikzcd}
\end{tabular}
\end{center}\caption*{Constructions of $\H(S)$}
\end{table}

There are a few things to observe here. First, we see that the arrows on the left diagram are the left adjoints to their corresponding arrows on the right diagram. Furthermore, the diagram on the right is a pullback square (taken in the category of $(\infty,1)$-categories) by the construction in \cite{BH17}. Additionally, note that the arrow $\ell_{\A^1}$ is exactly the localization defined in \cite{R15}, when we take our universes to be sufficiently small. Thus we conclude that 
\[
\H(S) = \ell_{\A^1} \Shv_\Nis (\Sm^{ft}/S)^\hyp \simeq \Shv_\Nis (\Sm^{ft}/S) \cap \PShv_{\A^1} (\Sm^{ft}/S)
\]
is an equivalence of $(\infty,1)$-categories.\\

Finally, we conclude with a nice fact about the above diagram (which we attribute to Tom Bachmann):

\begin{prop} In the figure above, the left square is a pushout of $(\infty,1)$-categories, and the right square is a pullback of $(\infty,1)$-categories.
\end{prop}

\begin{proof}
We can take the right square to be a pullback, by definition. This is still the case if we restrict to $\Pr^R$, the category of locally presentable $(\infty,1)$-categories with right adjoints as morphisms. In this category, the right square is a pullback.\\

There is additionally an anti-equivalence of categories $\Pr^R \xrightarrow{\cong} (\Pr^L)^\op$, which send right adjoints to left adjoints. Via this map, we can see that the left square is a pushout.
\end{proof}

Thus we can express the unstable  $\H(S)$ as a fibered coproduct in $\Pr^L$:

\[
\H(S) \simeq \PShv_{\A^1} (\Sm^{ft} / S) \amalg_{\PShv(\Sm^{ft}/S)} \Shv_{\Nis}(\Sm^{ft} / S).
\]


\begin{thebibliography}{8}

\bibitem[AE17]{AE16}
Benjamin Antieau and Elden Elmanto.
\newblock A primer for unstable motivic homotopy theory.
\newblock In {\em Surveys on recent developments in algebraic geometry},
  volume~95 of {\em Proc. Sympos. Pure Math.}, pages 305--370. Amer. Math.
  Soc., Providence, RI, 2017.

\bibitem[BH17]{BH17}
Tom Bachmann and Marc Hoyois.
\newblock Norms in motivic homotopy theory.
\newblock {\em arXiv preprint arXiv:1711.03061}, 2017.

\bibitem[Dug01]{D01}
Daniel Dugger.
\newblock Universal homotopy theories.
\newblock {\em Adv. Math.}, 164(1):144--176, 2001.

\bibitem[Lur04]{DAG}
Jacob Lurie.
\newblock {\em Derived algebraic geometry}.
\newblock ProQuest LLC, Ann Arbor, MI, 2004.
\newblock Thesis (Ph.D.)--Massachusetts Institute of Technology.

\bibitem[Lur09]{HTT}
Jacob Lurie.
\newblock {\em Higher topos theory}, volume 170 of {\em Annals of Mathematics
  Studies}.
\newblock Princeton University Press, Princeton, NJ, 2009.

\bibitem[MV99]{MV99}
Fabien Morel and Vladimir Voevodsky.
\newblock {${\bf A}^1$}-homotopy theory of schemes.
\newblock {\em Inst. Hautes \'{E}tudes Sci. Publ. Math.}, (90):45--143 (2001),
  1999.

\bibitem[Rob15]{R15}
Marco Robalo.
\newblock {$K$}-theory and the bridge from motives to noncommutative motives.
\newblock {\em Adv. Math.}, 269:399--550, 2015.

\bibitem[Voe98]{V98}
Vladimir Voevodsky.
\newblock {$\bold A^1$}-homotopy theory.
\newblock In {\em Proceedings of the {I}nternational {C}ongress of
  {M}athematicians, {V}ol. {I} ({B}erlin, 1998)}, number Extra Vol. I, pages
  579--604, 1998.

\end{thebibliography}
\end{document}